\documentclass{amsart}
\usepackage[utf8]{inputenc}
\usepackage[english]{babel}
\usepackage{blindtext}
\usepackage{amssymb}
\usepackage{enumitem}
\usepackage{amsthm}
\usepackage{amsmath}
\usepackage{amsfonts}
\usepackage{indentfirst}
\usepackage{comment}
\usepackage{mathtools}
\usepackage{csquotes}
\usepackage{setspace}
\usepackage[left=3cm,right=3cm,top=4.2cm,bottom=4.2cm, a4paper]{geometry}
\usepackage{biblatex}
\subjclass[2020]{Primary 60J10, 65C05, 68W20}
\newtheorem{thm}{Theorem}[section]

\newtheorem{prop}[thm]{Proposition}

\theoremstyle{definition}

\theoremstyle{remark}

\usepackage{hyperref}

\title[Cutoff for Cyclic Dynamics]{Cutoff Phenomenon for Cyclic Dynamics on Hypercube}\thanks{The revised version of this draft is published in the Markov Processes and Related Fields journal.}
\author{Keunwoo Lim}
\address{Department of Mathematical Sciences\\ Seoul National University}
\email{kwlim.snu@gmail.com}
\keywords{Irreversible Markov chain, mixing time, cutoff phenomenon}
\begin{document}
\maketitle
\begin{abstract}
The cutoff phenomena for Markovian dynamics have been observed and
rigorously verified for a multitude of models, particularly for Glauber-type
dynamics on spin systems. However, prior studies have barely considered irreversible chains. In this work, the cutoff phenomenon of
certain cyclic dynamics are studied on the hypercube $\Sigma_{n} =
Q^{V_{n}}$, where $Q = \{1, 2, 3\}$ and $V_{n} = \{1,...,n\}$. The main
feature of these dynamics is the fact that they are represented by an
irreversible Markov chain. Based on the couplings modified from the previous study of the cutoff phenomenon for the Curie-Weiss-Potts
model, a comprehensive proof is presented.
\end{abstract}
\section{Introduction}\label{1}
This work considers the mixing behavior of irreversible dynamics
on the hypercube $\Sigma_{n} = Q^{V_{n}}$, where $Q = \{1, 2, 3\}$ and
$V_{n} = \{1,...,n\}$. 
We consider the hypercube as the structure that assigns the color in $Q$ on each vertex in $V_{n}$. One of the widely known Markov chains is the discrete time
Glauber dynamics for the uniform measure on $\Sigma_{n}$. At each time
step, the vertex $v\in V_{n}$ is uniformly chosen. Then, we reassign the color of vertex $v$ uniformly on $Q$.
The mixing of these
dynamics is fully understood, and the sharp convergence exhibited is
defined as the cutoff phenomenon. \par
In this study, the result is extended to
the discrete time cyclic dynamics $(\sigma^{n}_{t})_{t=0}^{\infty}$ iterated by the following rule. At time $t+1$, the vertex $v\in V_{n}$ is uniformly chosen. 
Then, $\sigma^{n}_{t+1}$ is set as
\begin{equation*}
\sigma^{n}_{t+1}(w) =
\begin{cases}
\sigma^{n}_{t}(w)\text{ w.p. }1&\text{ if }w\neq v\\
\sigma^{n}_{t}(w)\text{ w.p. }1-p\,\,\,\text{and}\,\,\,\sigma^{n}_{t}(w)+1\text{ w.p. }p
&\text{ if }w=v,
\end{cases}
\end{equation*}
where $0<p<1$. Here, $\sigma^{n}_{t}(w)$ is denoted as the color of vertex $w$ on $\sigma^{n}_{t}$ and w.p. is an abbreviation of ``with probability." The color of each vertex is evaluated based on modular arithmetic modulo 3. The cutoff phenomenon described below is
proved. \par
The descriptions of the cutoff phenomenon are based on \cite{2}. Let the
total variance distance between the two probability distributions $\mu$
and $\nu$ on discrete state space $\mathcal{X}$ be defined as
\begin{equation*}
\Vert\mu-\nu\Vert_{\textnormal{TV}} =
\max_{A\subseteq\mathcal{X}}\,\vert\,\mu(A)-\nu(A)\,\vert.
\end{equation*}
Then, consider the Markov chain $(X_{t})$ on state space $\mathcal{X}$ with the
transition matrix $P$ and stationary distribution $\pi$. The maximal
distance $d(t)$ of the Markov chain $(X_{t})$ is defined as
\begin{equation*}
d(t)=\max_{x\in\mathcal{X}}\,\Vert P^{t}(x,\cdot)-
\pi\Vert_{\textnormal{TV}},
\end{equation*}
while the $\epsilon$-mixing time is defined as
\begin{equation*}
t_{\textnormal{mix}}(\epsilon)=\min\{t\,:\,d(t)\leq\epsilon\}.
\end{equation*}
The mixing time $t_{\textnormal{mix}}$ is denoted as
$t_{\textnormal{mix}}(\frac{1}{4})$ by convention.\par
For all $\epsilon\in(0,1)$, suppose that the sequence of Markov chains $
\{(X_{t}^{n})\}=(X^{1}_{t}),\,(X^{2}_{t}),\,\dots$
satisfies
\begin{equation*}
\lim_{n\to\infty}\frac{t^{(n)}_{\textnormal{mix}}(\epsilon)}{t^{(n)}_{\textnormal{mix}}(1-\epsilon)} = 1,
\end{equation*}
where $t^{(n)}_{\text{mix}}(\epsilon)$ is the $\epsilon$-mixing time of
the chain $(X^{n}_{t})$. Denote the mixing time of the $n$-th chain as $t_{\textnormal{mix}}^{(n)} =
t_{\textnormal{mix}}^{(n)}(\frac{1}{4})$, and the maximal distance as $d^{(n)}(t)$. Then, these Markov chains shows a sharp
decrease in the total variance distance from $1$ to $0$ close to the
mixing time. It is said that this sequence exhibits the cutoff
phenomenon. Further, it is said to have a window of size $O(w_{n})$ if $\lim_{n\to\infty}\big(w_{n}/{t_{\textnormal{mix}}^{(n)}}\big)=0$,
\begin{gather*}
\lim_{\alpha\to-
\infty}\liminf_{n\to\infty}d^{(n)}\big(t_{\textnormal{mix}}^{(n)}+\alpha
w_{n}\big) = 1\quad\text{and}\quad
\lim_{\alpha\to\infty}\limsup_{n\to\infty}d^{(n)}\big(t_{\textnormal{mix}}^{
(n)}+\alpha w_{n}\big) = 0.
\end{gather*}
The cutoff phenomenon was first observed in card shuffling, as
demonstrated in \cite{3}. Since then, the cutoff phenomena for Markovian dynamics have been observed and rigorously verified with a
multitude of models. In recent times, there have been several
breakthroughs in the verification of cutoff phenomena for Glauber-type
dynamics on spin systems. For example, the cutoff phenomenon for the
Glauber dynamics on the Curie-Weiss model, which corresponds to the mean-field
Ising model, is proven in \cite{4} for the high temperature regime.
This work has been further generalized to \cite{1}, where Glauber dynamics
for Curie-Weiss-Potts model have been considered. These two outcomes were
considered on the mean-field model defined on the complete graph, where
geometry is irrelevant.\par
On the other hand, the cutoff phenomenon for the spin system on the
lattices were more complicated. The first development was achieved for the
Ising model on the lattice in \cite{5}, and in \cite{6}, it was extended
to a general spin system in the high temperature regime. In \cite{7}, a
novel method called ``information percolation'' was developed, and the
cutoff for the Ising model with a precise window size was obtained. This
information percolation method has also been successfully applied to
Swendsen-Wang dynamics for the Potts model and to Glauber dynamics for the
random-cluster model in \cite{8} and \cite{9}, respectively.\par
In the present work, the uniform measure, which corresponds to infinite
temperature spin systems, is considered. From this perspective, the
proposed model is simpler than existing models in which finite temperature
has been considered. However, this model has a critical difference in that
the dynamics being considered are irreversible. We emphasize here that the cutoff phenomenon for the irreversible chains are known only for few models, e.g., non-backtracking random walks on sparse random graphs \cite{10}.
\subsection{Main Result}
Theorem \ref{1.1} presents the cutoff phenomenon of the cyclic dynamics
considered herein and the main result of the current article.
\begin{thm}\label{1.1}
The cyclic dynamics defined on $\Sigma_{n}$ with probability $0<p<1$
exhibit cutoff at mixing time
\begin{equation*}
t(n) =\frac{1}{3p}n\log n
\end{equation*}
with a window of size $O(n)$. 
\end{thm}
As the theorem can be similarly proved for all $0 < p < 1$, the proof is
presented for $p=\frac{1}{2}$. In Section \ref{2}, the notations are set
and the contractions of the proportion chain are provided. The proof of
the lower bound of the cutoff is then presented. Section \ref{3} analyzes the coalescence of the proportion and basket chains.
Following this, the upper bound of the cutoff is proved.\par
The dynamics considered in this article is a Glauber-type (but asymmetric) dynamics on Curie-Weiss-Potts model with three spins at infinite temperature. The cutoff for usual symmetric Glauber dynamics on Curie-Weiss-Potts model has been verified for all the high temperature regime in \cite{1}. For the asymmetric dynamics, the metastability for all the low temperature regime has been thoroughly analyzed in \cite{11} for the three spin case. It is widely believed that the asymmetric dynamics also exhibits cutoff phenomenon at all the entire high temperature regime, but the proof is missing at this moment; the current article investigated the special case of the last problem.\par
The structure of the proof is similar to the case of the cutoff phenomenon of the Glauber dynamics for the Curie-Weiss-Potts model in high temperature regime presented in \cite{1}. The convergence to the stationary distribution is obtained by successively coalescing the proportion chains and the basket chains with the coupling methods. The major difference with the previous method is the construction of the appropriate couplings to deal with the asymmetric nature of the irreversible dynamics. They are based on the couplings introduced in \cite{1}, but more sophisticated constructions are needed in cases where symmetry is starting to break. \par
The proof cannot be generalized to the cases where the number of colors are larger than three. One of the obstructions is the proof of the Proposition \ref{2.2}, which presents the convergence of the proportion chain to stationary distribution in $\ell^{2}$-norm. The computation is simplified only when the number of colors are three. \par

We remark that the Glauber dynamics for the uniform measure that corresponds to the current model exhibits the cutoff phenomenon. It is proven in \cite{2} that this reversible dynamics exhibits the cutoff at $\frac{1}{2}n\log n$ with a window of size $O(n)$. 
For that reversible case, the spectral analysis can be applied to obtain the upper bound (see \cite[Chapter 12]{2}). In particular, a direct relationship between the eigenvalues of the transition matrix and the bound of the total variation distance is crucially used. For our irreversible case, we are not able to use spectral analysis and the proof becomes more complex.\par
Note that when $p>\frac{2}{3}$, the mixing time of the cyclic dynamics considered in this article is smaller than the mixing time of the Glauber dynamics defined above. It shows that the irreversible chain can converge faster into uniform stationary distribution than reversible chain. This study is the part of an attempt to provide the theoretical background in applying the irreversible Markov chains to Markov chain Monte Carlo methods which is believed to be faster than the reversible one.
\par
\section{Lower Bound}\label{2}
This section presents the lower bound of the cutoff and the proof is given in Section \ref{2.14}.
The proof is based on the analyses of the statistical properties of cyclic
dynamics described
in Section \ref{2.12} and the features evaluated on a stationary distribution in Section \ref{2.13}. Prior to the
proof, the notations are set, and the proportion chain used
throughout this paper is defined.
\subsection{Preliminaries}\label{2.21}
Denote the cyclic dynamics on $\Sigma_{n}$ as $(\sigma^{n}_{t})_{t=0}^{\infty}$, and eliminate $n$ for simplicity. When the cyclic dynamics $(\sigma_{t})$ begin at state
$\sigma_{0}$, denote the probability measure as
$\mathbb{P}_{\sigma_{0}}$, and the expectation with respect to the
probability measure as $\mathbb{E}_{\sigma_{0}}$. \par Then, consider the vector $s\in\mathbb{R}^{3}$ and let its $i$-th element as $s^{i}$. The $\ell^{p}$-norm of the vector $s$ is denoted as
$\Vert s \Vert_{p}$. Denote the vector $(\frac{1}{3}, \frac{1}{3}, \frac{1}{3})\in\mathbb{R}^{3}$
as $\bar{e}$, and let $\hat{s} = s - \bar{e}$. Consider the $3\times3$
matrix $\textbf{Q}$, and let
$\textbf{Q}^{i,k}$ be the ${(i,k)}$ element of matrix $\textbf{Q}$. Let $\textbf{Q}^{i}$ be its $i$-th row. For $\rho>0$, the subsets of
$\mathbb{R}^{3}$ are denoted as
\begin{alignat*}{2}
&\mathcal{S}=\big\{\,x \in \mathbb{R}^{3}_{+}:\Vert x
\Vert_{1}=1\,\big\},&&
\mathcal{S}_{n}=\mathcal{S}\cap\frac{1}{n}\mathbb{Z}^{3}, \\
&\mathcal{S}^{\rho} = \big\{\,s\in\mathcal{S}:\Vert
\hat{s}\Vert_{\infty}<\rho\,\big\}, &&\mathcal{S}_{n}^{\rho} =
\mathcal{S}^{\rho}\cap\frac{1}{n}\mathbb{Z}^{3},\\ &\mathcal{S}^{\rho +} =
\big\{\,s\in\mathcal{S}: s^{k}<\frac{1}{3}+\rho,\, 1 \leq k \leq 3\,
\big\},\quad&&\mathcal{S}_{n}^{\rho +} = \mathcal{S}^{\rho
+}\cap\frac{1}{n}\mathbb{Z}^{3}.
\end{alignat*}
Now, the \textit{proportion chain} $(S_{t})_{t=0}^{\infty}$ of the cyclic dynamics
$(\sigma_{t})_{t=0}^{\infty}$ is defined as
\begin{equation*}
S_{t}= \big(\,S_{t}^{1},\, S_{t}^{2},\, S_{t}^{3}\,\big),
\end{equation*}
where
\begin{equation*}
S_{t}^{k} = \frac{1}{n}\sum_{v \in V_{n}}\,\textbf{1}_{\{ \sigma_{t}(v) =
k\}} \quad k = 1,\, 2,\, 3.
\end{equation*}
Then, the proportion chain $(S_{t})$ is also a Markov chain on state space
$\mathcal{S}_{n}$ with jump probability
\begin{equation*}
\big(\,S_{t+1}^{1},\, S_{t+1}^{2},\, S_{t+1}^{3}\,\big)=
\begin{cases}
\big(\,S_{t}^{1},\, S_{t}^{2},\, S_{t}^{3}\,\big)&\text{ w.p.
}\frac{1}{2}\vspace{3pt} \\
\big(\,S_{t}^{1}-\frac{1}{n},\, S_{t}^{2}+\frac{1}{n},\,
S_{t}^{3}\,\big)&\text{ w.p. }\frac{1}{2}S_{t}^{1}\vspace{3pt} \\
\big(\,S_{t}^{1},\, S_{t}^{2}-\frac{1}{n},\,
S_{t}^{3}+\frac{1}{n}\,\big)&\text{ w.p. }\frac{1}{2}S_{t}^{2}\vspace{3pt}
\\
\big(\,S_{t}^{1}+\frac{1}{n},\, S_{t}^{2},\, S_{t}^{3}-
\frac{1}{n}\,\big)&\text{ w.p. }\frac{1}{2}S_{t}^{3}. \\
\end{cases}
\end{equation*}
This formulation is well-defined on $\mathcal{S}_{n}$, because if $S_{t}^{i}=0$ for any $i \in \{1, 2, 3\}$, then the probability of
$S_{t}^{i}$ decreasing in the next step is zero.
\subsection{Statistical Properties of the Chain}\label{2.12}
This section describes the statistical properties of the
proportion chain used in the proof. In particular, the $\ell^{2}$-norm of
$\hat{S}_{t}$ and the variance of $S_{t}$ are analyzed.
\begin{prop}\label{2.2}
Proportion chain $(S_{t})$ of the cyclic dynamics $(\sigma_{t})$ has the following $\ell^{2}$-norm contraction that depends on $n$:
\begin{equation*}
\mathbb{E}_{\sigma_{0}}\Vert \hat{S}_{t}\Vert_{2}^{2} = \Big(\,1-
\frac{3}{2n}\,\Big)^{t}\,\Vert \hat{S}_{0}\Vert_{2}^{2} \,+\,O\,\Big(\,\frac{1}{n}\,\Big).
\end{equation*}
\end{prop}

This shows the contraction on the expectation of $\ell^{2}$-norm of
$\hat{S}_{t}$. In Proposition \ref{2.8}, this result is used to evaluate the expectation of $\hat{S}_{t}$ at the certain time. Next, the semi-synchronized coupling that contracts the norm between the two proportion chains is defined. It is
similar to the synchronized coupling of \cite{1}, but it has more
comprehensive cases.
\subsubsection{Semi-Synchronized Coupling} Consider the two cyclic
dynamics $(\sigma_{t})$ and $(\tilde{\sigma}_{t})$ starting from
$\sigma_{0}$, $\tilde{\sigma}_{0}$. Denote their proportion chains as
$(S_{t})$ and $(\tilde{S}_{t})$. At time $t+1$, the \textit{semi-synchronized
coupling} for the case of $S^{1}_{t}\geq\tilde{S}^{1}_{t}$,
$S^{2}_{t}\leq \tilde{S}^{2}_{t}$, $S^{3}_{t} \leq \tilde{S}^{3}_{t}$ is
defined as follows:
\begin{enumerate}
\setlength\itemsep{3pt}
\item Choose the colors $(\,I_{t+1},\, \tilde{I}_{t+1}\,)$ based
on the probability as stated below:
\begin{equation*}
(\,I_{t+1},\, \tilde{I}_{t+1}\,)\,=
\begin{cases}
(\,1,\,1\,)\text{ w.p. }
\tilde{S}_{t}^{1}\\
(\,2,\,2\,)\text{ w.p. }
S_{t}^{2}\\
(\,3,\,3\,)\text{ w.p. }
S_{t}^{3}\\
(\,1,\,2\,)\text{ w.p. }
\tilde{S}_{t}^{2}-S_{t}^{2}\\
(\,1,\,3\,)\text{ w.p. }
\tilde{S}_{t}^{3}-S_{t}^{3}.
\end{cases}
\end{equation*}
\item Choose the colors $(\,J_{t+1},\, \tilde{J}_{t+1}\,)$
depending on $(\,I_{t+1},\, \tilde{I}_{t+1}\,)$ based on the probability
as stated below:
\begin{itemize}
\setlength\itemsep{3pt}
\item$(\,I_{t+1},\, \tilde{I}_{t+1}\,)\, =\, (\,1,\,1\,)\Rightarrow
(\,J_{t+1},\, \tilde{J}_{t+1}\,)$ is $(\,1,\, 1\,)$ w.p. $\frac{1}{2}$,
and is $(\,2,\,2\,)$ w.p. $\frac{1}{2}.$
\item$(\,I_{t+1},\, \tilde{I}_{t+1}\,)\, =\, (\,2,\,2\,)\Rightarrow(\,J_{t+1},\, \tilde{J}_{t+1}\,)$ is $(\,2,\,2\,)$ w.p. $\frac{1}{2}$, and is
$(\,3,\,3\,)$ w.p. $\frac{1}{2}$.
\item$(\,I_{t+1},\, \tilde{I}_{t+1}\,) \,=\, (\,3,\,3\,)\Rightarrow(\,J_{t+1},\, \tilde{J}_{t+1}\,)$ is $ (\,3,\,3\,)$ w.p. $\frac{1}{2}$,
and is $(\,1,\,1\,)$ w.p. $\frac{1}{2}$.
\item$(\,I_{t+1},\, \tilde{I}_{t+1}\,)\, =\, (\,1,\,2\,)\Rightarrow(\,J_{t+1},\, \tilde{J}_{t+1}\,)$ is $(\,1,\,3\,)$ w.p. $\frac{1}{2}$,
and is $(\,2,\,2\,)$ w.p. $\frac{1}{2}$.
\item$(\,I_{t+1},\, \tilde{I}_{t+1}\,)\, =\, (\,1,\,3\,)\Rightarrow(\,J_{t+1},\, \tilde{J}_{t+1}\,)$ is $(\,1,\,1\,)$ w.p. $\frac{1}{2}$, and is
$(\,2,\,3\,)$ w.p. $\frac{1}{2}$.
\end{itemize}
\item Choose a vertex that has the color ${I}_{t+1}$ in
${\sigma}_{t}$ uniformly. Then, change its color to ${J}_{t+1}$ in
${\sigma}_{t+1}$.
\item Choose a vertex that has the color $\tilde{I}_{t+1}$
in $\tilde{\sigma}_{t}$ uniformly. Then, change its color to $\tilde{J}_{t+1}$ in
$\tilde{\sigma}_{t+1}$.
\end{enumerate}
Semi-synchronized coupling for the other cases can be defined in a similar
manner. Let $\mathbb{P}^{SC}_{\sigma_{0},\tilde{\sigma}_{0}}$ be the
underlying probability measure of this coupling, and
$\mathbb{E}^{SC}_{\sigma_{0},\tilde{\sigma}_{0}}$ be the expectation with
respect to the underlying probability measure. This coupling is
constructed to obtain the following $\ell^{1}$-contraction result.
\begin{prop}\label{2.3}
Consider the semi-synchronized coupling of two cyclic dynamics
$(\sigma_{t})$ and $(\tilde{\sigma}_{t})$. Then, the following equation holds:
\begin{equation*}
\mathbb{E}^{SC}_{\sigma_{0},\tilde{\sigma}_{0}}\Vert S_{t}-
\tilde{S}_{t}\Vert_{1}
\leq\Big(1-\frac{1}{2n}\Big)^{t}\,\Vert S_{0}-\tilde{S}_{0}\Vert_{1}.
\end{equation*}
\end{prop}

The following propositions bound the variance of the proportion chain value at time $t$ from the
contraction of the norm between two proportion chains. The following theorem presents the
relation between the variance and the contraction. Its only
difference from \cite[Lemma 2.4]{1} is the coefficient $c>1$.
\begin{prop}\cite[Lemma 2.4]{1}\label{2.4}
Consider the Markov chain $(Z_{t})$ taking values in $\mathbb{R}^{d}$. When $Z_{0} = z$, let $\mathbb{P}_{z}$ and
$\mathbb{E}_{z}$ be its probability measure and expectation, respectively. If there exists $0<\rho<1$ and $c>1$ that satisfies
$\Vert\,\mathbb{E}_{z}[Z_{t}]-
\mathbb{E}_{\tilde{z}}[Z_{t}]\,\Vert_{2}\leq c \rho^{t}\,\Vert z-
\tilde{z}\Vert_{2}$
for every pairs of starting point $(z, \tilde{z}),$
then
\begin{equation*}
v_{t}\,=\,\sup_{z_{0}}\, \mathbb{V}ar_{z_{0}}\,(Z_{t})\,=\,\sup_{z_{0}}\,\mathbb{E}_{z_{0}}\,\Vert\, Z_{t}-
\mathbb{E}_{z_{0}}Z_{t}\,\Vert_{2}^{2}
\end{equation*}
satisfies
\begin{equation*}
v_{t}\, \leq\, c^{2}\, v_{1}\,\min\big\{\,t, \big(1-\rho^{2}\big)^{-
1}\big\}.
\end{equation*}
\end{prop}
\begin{prop}\label{2.5}
For the cyclic dynamics $(\sigma_{t})$ starting from $\sigma_{0}$ and all $t \geq 0$,
\begin{equation*}
\mathbb{V}ar_{\sigma_{0}}(S_{t}) = O\big(n^{-1}\big).
\end{equation*}
\end{prop}

\subsection{Statistics of Stationary Distribution}\label{2.13}
This section presents the proof that $\mu_{n}$ is the stationary
distribution of the cyclic dynamics. Here, $\mu_{n}$ is the uniform probability measure on state space $\Sigma_{n}$, i.e.
\begin{equation*}
\mu_{n}(\sigma) = \frac{1}{3^{n}}\quad \forall\, \sigma\in\Sigma_{n}.
\end{equation*}
The underlying probability measure, expectation, and variance are denoted
as $\mathbb{P}_{\mu_{n}}$, $\mathbb{E}_{\mu_{n}}$, and
$\mathbb{V}ar_{\mu_{n}}$, respectively.
First, recall \cite[Corollary 1.17]{2}, which describes the stationary
distribution in the irreducible Markov chain.
\begin{prop}\cite[Corollary 1.17]{2} Let $P$ be the transition matrix of the irreducible Markov chain. Then, there exists a unique stationary
distribution of the chain.
\end{prop}
Then, we introduce the product chain suggested in \cite[Section 12.4]{2}.
For $j = 1,\dots,n$, consider the irreducible Markov chain $(Z^{j}_{t})$ on state
space $\mathcal{X}_{j}$ with transition matrix $P_{j}$. Let
$w=(w_{1},\dots,w_{n})$ be a probability distribution of $\{1,\dots,n\}$,
where $0<w_{j}<1$. Define the product chain on state space $\mathcal{X} =
\mathcal{X}_{1}\times\cdots\times\mathcal{X}_{n}$ with transition matrix
$P$ that has the transition probability as
\begin{equation*}
P(x, y) = \sum_{j=1}^{n}w_{j}P_{j}(x_{j},y_{j})\prod_{i:i\neq
j}\textbf{1}_{\{x_{i}=y_{i}\}}
\end{equation*}
for any two states $x = (x_{1},\dots,x_{n}),\, y =
(y_{1},\dots,y_{n})\in\mathcal{X}$.
For the functions $f^{(1)},\dots,f^{(n)}$, where $f^{(j)}\colon
\mathcal{X}_{j}\to\mathbb{R}$, define the product on $\mathcal{X}$
as
\begin{equation*}
f^{(1)} \otimes f^{(2)} \otimes\cdots\otimes
f^{(n)}(x_{1},\dots,x_{n})=f^{(1)}(x_{1})\cdots f^{(n)}(x_{n}).
\end{equation*}
\begin{prop}\label{2.16} Consider the product chain of the Markov chains $(Z^{1}_{t}),\dots,(Z^{n}_{t})$ as above. For $j
= 1,\dots,n$, let $\pi^{(j)}$ be the stationary distribution of the chain $(Z^{j}_{t})$. Then, $
\pi^{(1)} \otimes\pi^{(2)}\otimes\cdots\otimes \pi^{(n)}$ is the
stationary distribution of the product chain.
\end{prop}

\begin{prop}\label{2.6}
The probability measure $\mu_{n}$ is a unique stationary distribution of
the cyclic dynamics $(\sigma^{n}_{t})$.
\end{prop}

Now, define the function $S\colon\Sigma_{n}\to\mathcal{S}_{n}$ as $S(\sigma) = \big(\,S^{1}(\sigma), S^{2}(\sigma), S^{3}(\sigma)\, \big)
$, where
\begin{equation*}
S^{k}(\sigma) = \frac{1}{n}\sum_{v \in V_{n}}\,\textbf{1}_{\{ \sigma(v) =
k\}} \quad k = 1,\, 2,\, 3.
\end{equation*}
Consider the case where the element $\sigma\in\Sigma_{n}$ is distributed according to the
probability distribution $\mu_{n}$. Because the element $\sigma$ is uniformly
distributed, $n\cdot S^{k}(\sigma)$ can be interpreted as the sum of $n$
independent variables having the value $1$ with $\frac{1}{3}$ and $0$ with
$\frac{2}{3}$. Thus, $n\cdot S^{1}(\sigma),n\cdot S^{2}(\sigma),n\cdot
S^{3}(\sigma)\sim \text{Bin}(n,\, \frac{1}{3})$ and
\begin{gather*}
\mathbb{E}_{\mu_{n}}S(\sigma)=\big(\,\mathbb{E}_{\mu_{n}}S^{1}(\sigma),\,\mathbb{E}_{\mu_{n}}S^{2}(\sigma),\,\mathbb{E}_{\mu_{n}}S^{3}(\sigma)\,\big
) = \Big(\,\frac{1}{3},\,\frac{1}{3},\,\frac{1}{3}\,\Big)\\
\mathbb{V}ar_{\mu_{n}}S(\sigma) = \mathbb{V}ar_{\mu_{n}}S^{1}(\sigma)+
\mathbb{V}ar_{\mu_{n}}S^{2}(\sigma)+\mathbb{V}ar_{\mu_{n}}S^{3}(\sigma) =
\frac{2}{3n}.
\end{gather*}
\subsection{Proof of Lower Bound}\label{2.14}
This section presents the proof of the lower bound of the mixing time. The
proof of the proposition is based on \cite[Section 4.1]{1}. Denote
$t(n)=\frac{2}{3}n \log n$ and $t_{\gamma}(n)=\frac{2}{3}n\log n+\gamma
n$. The main
principle is to compare the probability of the event
$\{\Vert\hat{S}_{t_{\gamma}(n)}\Vert_{2}<\frac{r}{\sqrt{n}}\}$ under the
probability measure $\mathbb{P}_{\sigma_{0}}$ and under $\mu_{n}$. 
\begin{prop}\label{2.8}
Consider the cyclic dynamics $(\sigma_{t})$ and a fixed constant
$\epsilon>0$. Then, for all sufficiently large values of $n$, there exists
a sufficiently large $-\gamma > 0$ that satisfies $t^{(n)}_{\textnormal{mix}}(1-\epsilon) \geq t_{\gamma}(n).$
\end{prop}
\begin{proof}
Note that 
\begin{equation*}
\Big(\,1-\frac{1}{x}\,\Big)^{x-1}\,>\,e^{-1}\,>\,\Big(\,1-
\frac{1}{x}\,\Big)^{x}
\end{equation*}
holds for all $x>1$.
Set the constant $0<\rho<\frac{2}{3}$. Then, choose a configuration $\sigma_{0}\in
\Sigma_{n}$ such that it satisfies
$\rho<\Vert\hat{S_{0}}\Vert_{2}$. Then, for $t \leq t_{\gamma}(n)$,
\begin{equation*}
\mathbb{E}_{\sigma_{0}}\Vert\hat{S}_{t}\Vert_{2}^{2}\, =\, \Big(1-
\frac{3}{2n}\Big)^{t}\,\Vert\hat{S}_{0}\Vert_{2}^{2} \,+\,O\,\Big(\,\frac{1}{n}\,\Big)\,\geq\,\frac{1}{n}\,e^{-\gamma}
\end{equation*}
holds for all sufficiently large $n$ and $-\gamma>0$ values depending on
$\rho$.\par
In addition, because $\mathbb{V}ar_{\sigma_{0}}(S_{t}) = O(n^{-1})$ by
Proposition \ref{2.5}, $\mathbb{V}ar_{\sigma_{0}}(\hat{S_{t}}) = O(n^{-
1})$ holds. It leads that
\begin{equation*}
\big(\,\mathbb{E}_{\sigma_{0}}\,\Vert\hat{S_{t}}\Vert_{2}\,\big)^{2}\,
\geq\,\mathbb{E}_{\sigma_{0}}\,\Vert\hat{S}_{t}\Vert_{2}^{2} -
\mathbb{V}ar_{\sigma_{0}}\big(\,\hat{S_{t}}\,\big)\,\geq\,\frac{1}{n}\,e^{
-\gamma}-O\,\big(\,n^{-1}\,\big),
\end{equation*}
and it implies that for all sufficiently large $n$ and $-\gamma$ values,
\begin{equation*}
\mathbb{E}_{\sigma_{0}}\,\Vert\hat{S_{t}}\Vert_{2}\,\geq\,\frac{1}{\sqrt{n
}}\,e^{-\frac{\gamma}{3}}.
\end{equation*}\par
Therefore, for $0<r<e^{-\frac{\gamma}{3}}$ and $t \leq t_{\gamma}(n)$, by Chebyshev's inequality and
Proposition \ref{2.5},
\begin{align*}
\mathbb{P}_{\sigma_{0}}\Big(\,\Vert\hat{S_{t}}\Vert_{2}\,<\,\frac{r}{\sqrt{n}}\,\Big) \leq&
\,\mathbb{P}_{\sigma_{0}}\,\Big(\,\mathbb{E}_{{\sigma}_{0}}\,\Vert\hat{S_{t
}}\Vert_{2}\,-\,\Vert\hat{S_{t}}\Vert_{2}\, >\,
\frac{1}{\sqrt{n
}}\,e^{-\frac{\gamma}{3}}\,-
\,\frac{r}{\sqrt{n}}\,\Big) \\ \leq&\,
\frac{\mathbb{V}ar_{\sigma_{0}}\,\big(\,\hat{S_{t}}\,\big)}{\big(\,\frac{1}{\sqrt{n
}}\,e^{-\frac{\gamma}{3}}-
\frac{r}{\sqrt{n}}\,\big)^{2}} = O\big(\,(\,e^{-\frac{\gamma}{3}}-r\,)^{-
2}\,\big).
\end{align*}
It follows that
\begin{equation*}
\lim_{\gamma\rightarrow-
\infty}\,\limsup_{n\rightarrow\infty}\,\mathbb{P}_{\sigma_{0}}\Big(\,\Vert\hat{S}_{t_{
\gamma}(n)}\Vert_{2}\,<\,\frac{r}{\sqrt{n}}\,\Big) = 0.
\end{equation*} \par
Now, consider the cyclic dynamics $(\sigma_{t})$ where $\sigma_{0}$
follows the probability distribution $\mu_{n}$. By the properties in
Section \ref{2.13} and application of Chebyshev's inequality,
\begin{equation*}
\mu_{n} \Big(\,
\Vert\hat{S}_{t}\Vert_{2}\,<\,\frac{r}{\sqrt{n}}\,\Big)\,\geq\,1\,-
\,\frac{O(1)}{r^{2}}
\end{equation*}
holds for all $t \geq 0.$
It can be concluded that for all $r>0$,
\begin{equation*}
\lim_{\gamma\rightarrow-
\infty}\,\liminf_{n\rightarrow\infty}\,d^{(n)}(t_{\gamma}(n))\,\geq\,1\,-
\,\frac{O(1)}{r^{2}}.
\end{equation*}
Letting $r \rightarrow \infty$, the proof is complete.
\end{proof}
\section{Upper Bound}\label{3}
This section presents the proof of the upper bound of the mixing time. In
\cite{1}, it is observed that the cutoff of the upper bound for the Glauber dynamics essentially follows from the precise bound on the
coalescing time of the two basket chains. In Section \ref{3.21}, semi-coordinatewise coupling is used to analyze the coalescence of the
proportion chains. In Section \ref{3.22}, basket chain is introduced,
and basketwise coupling is used to analyze the coalescence of the basket
chains. Based on those analyses, the upper bound of the cutoff is
obtained, as presented in Section \ref{3.44}. \par
The following proposition describes the distribution of the proportion chain at time $t(n)$.
It bounds the $\ell^{\infty}$-norm between $S_{t}$ and $\bar{e}$ over $n^{-
\frac{1}{2}}$ scale.
\begin{prop}\label{3.1}
Consider the cyclic dynamics $(\sigma_{t})$ starting at $\sigma_{0}$ and
its proportion chain $(S_{t})$. For all $r > 0$ and
$\sigma_{0}\in\Sigma_{n}$, it holds that
\begin{equation*}
\mathbb{P}_{\sigma_{0}}\,\Big(\,S_{t(n)}\,\notin\,
\mathcal{S}^{\frac{r}{\sqrt{n}}}\,\Big) = O\,\big(\,r^{-1}\,\big).
\end{equation*}
\end{prop}
\subsection{Coalescing Proportion Chains}\label{3.21}
For the two cyclic dynamics $(\sigma_{t})$ and $(\tilde{\sigma}_{t})$, the
proportion chains $S_{t}$ and $\tilde{S}_{t}$ are made to coalesce with
high probability. First, $S_{t}-\bar{e}$ is bound with the $n^{-
\frac{1}{2}}$ scale. Then, $S_{t}$, $\tilde{S}_{t}$ is matched via
coupling under certain condition.
\subsubsection{Preliminaries}
Here, the two well-known theorems used in the current section are
introduced.
\begin{prop} \cite[Lemma 2.1 (2)]{1}\label{3.2}
Consider the discrete time process $(X_{t})_{t\geq0}$ adapted to
filtration $(\mathcal{F}_{t})_{t\geq 0}$ that starts at
$x_{0}\in\mathbb{R}$. Let the underlying probability measure as
$\mathbb{P}_{x_{0}}$, and let $\tau_{x}^{+}=\inf\{t:X_{t}\geq x\}.$ Then,
if the process $(X_{t})$ satisfies the below two conditions, the following statement
holds:
\begin{enumerate}[label=(\alph*)]
\setlength\itemsep{3pt}
\item$\exists\,\delta\geq0 \colon\mathbb{E}_{x_{0}}\,[\,X_{t+1}-
X_{t}\,\vert\,\mathcal{F}_{t}\,]\,\leq\,-\delta$ on $ \{\,X_{t}\geq0\,\}$
for all $t\geq 0.$
\item$\exists\, R > 0\colon \vert\, X_{t+1}-X_{t}\,\vert\,\leq\, R,\,
\forall\, t\geq0.$
\end{enumerate}
If $x_{0}\leq0$, then for $x_{1}>0$ and $t_{2}\geq0,$
\begin{equation*}
\mathbb{P}_{x_{0}}\,\big(\,\tau_{x_1}^{+}\,\leq\,{t_{2}}\,\big)\,\leq\,2\,
\exp\Big\{-\frac{(x_{1}-R)^{2}}{8t_{2}R^{2}}\Big\}.
\end{equation*}
\end{prop}
\begin{prop}\cite[Lemma 2.3]{1}\label{3.4}
Suppose that the non-negative discrete time process $(Z_{t})_{t\geq0}$
adapted to $(\mathcal{G}_{t})_{t\geq0}$ is a supermartingale. Let $N$ be a
stopping time. If $(Z_{t})$ satisfies the below three conditions:
\begin{enumerate}[label=(\alph*)]
\setlength\itemsep{3pt}
\item$Z_{0} = z_{0}$
\item$\vert\, Z_{t+1}-Z_{t}\,\vert \,\leq\, B$
\item$\exists \sigma>0$ such that $
\mathbb{V}ar\,(\,Z_{t+1}\,\vert\,\mathcal{G}_{t}\,)\,>\,\sigma^{2}$ on the
event $\{\,N>t\,\}$,
\end{enumerate}
and $u\,>\,4B^{2}\,/\,(3\sigma^{2})$, then
$\mathbb{P}_{z_{0}}\,(N>u)\,\leq\,4z_{0}\,/\,(\sigma\sqrt{u})$.
\end{prop}
\subsubsection{Restriction of Proportion Chain}
\begin{prop}\label{3.3}
Consider the cyclic dynamics $(\sigma_{t})$ and its proportion chain
$(S_{t})$. For the fixed constant $r_{0}>0$ and $\gamma>0$, there exist
$C, c>0$ that satisfies the following statement: \\For all sufficiently
large $n$ and $r>\max\,\{\,3r_{0},2\,\}$, let $t = \gamma n$, $\rho_{0} =
\frac{r_{0}}{\sqrt{n}}$ and $\rho = \frac{r}{\sqrt{n}}$. Then,
\begin{equation*}
\mathbb{P}_{\sigma_{0}}\,\big(\,\exists\, 0\leq u \leq t: S_{u} \notin
\mathcal{S}_{n}^{\rho+}\,\big)\,\leq\, Ce^{-cr^{2}}
\end{equation*}
holds for all $\sigma_{0}$ such that $S_{0}\in \mathcal{S}_{n}^{\rho_{0}+}.$
\end{prop}

\subsubsection{Semi-Coordinatewise Coupling}
This section introduces the semi-coordinatewise coupling of two cyclic
dynamics $(\sigma_{t})$ and $(\tilde{\sigma}_{t})$. This type of coupling
is used in Proposition \ref{3.14} to prove that $\Vert S_{t}-
\tilde{S}_{t}\Vert_{1}$ is a supermartingale. \par
First, semi-independent and coordinatewise coupling are defined as in
\cite{1}. For the two probability distributions $\nu$ and $\tilde{\nu}$ on
$\Omega=\{1,2,3\}$ and $i\in\{1,2,3\}$, $\{i\}$-semi-independent coupling
of $\nu$ and $\tilde{\nu}$ is a pair of random variables $(X,\tilde{X})$
on $\Omega\times\Omega$ as follows:
\begin{enumerate}
\setlength\itemsep{3pt}
\item Pick $U$ uniformly on [0, 1].
\item If $U\leq\min\big(\,\nu(i), \tilde{\nu}(i)\,\big)$, choose $(X,
\tilde{X})$ as $(i, i)$.
\item If $U>\min\big(\,\nu(i), \tilde{\nu}(i)\,\big)$,
choose $X$ and $\tilde{X}$ independently according to the following rules:
In the case of $X$, if $U<\nu(i)$, choose $i$. Otherwise, choose $i+1$
with probability $\frac{\nu(i+1)}{\nu(i+1)+\nu(i+2)}$, and choose $i+2$
with probability $\frac{\nu(i+2)}{\nu(i+1)+ \nu(i+2)}.$ In the case of
$\tilde{X}$, if $U<\tilde{\nu}(i)$, choose $i$. Otherwise, choose $i+1$
with probability
$\frac{\tilde{\nu}(i+1)}{\tilde{\nu}(i+1)+\tilde{\nu}(i+2)}$, and choose
$i+2$ with probability
$\frac{\tilde{\nu}(i+2)}{\tilde{\nu}(i+1)+\tilde{\nu}(i+2)}.$
\end{enumerate}
It is evident from the construction that random variables $X$ and
$\tilde{X}$ follow the distributions $\nu$ and $\tilde{\nu}$,
respectively.
Now, the $\{i\}$-coordinatewise coupling of two cyclic dynamics is
defined as follows:
\begin{enumerate}
\setlength\itemsep{3pt}
\item Choose two colors $I_{t+1}$ and $\tilde{I}_{t+1}$ based on
$\{i\}$-semi-independent coupling of $S_{t}$ and $\tilde{S}_{t}.$
\item Choose two colors $J_{t+1}$ and $\tilde{J}_{t+1}$ based on
$\{i\}$-semi-independent coupling, in turn based on $\nu$ and
$\tilde{\nu}$. Here, $\nu(I_{t+1})=\frac{1}{2}$,
$\nu(I_{t+1}+1)=\frac{1}{2}$, $\nu(I_{t+1}+2)=0$,
$\tilde{\nu}(\tilde{I}_{t+1})=\frac{1}{2}$,
$\tilde{\nu}(\tilde{I}_{t+1}+1)=\frac{1}{2}$ and
$\tilde{\nu}(\tilde{I}_{t+1}+2)=0$.
\item Uniformly choose the vertex of color $I_{t+1}$ in $\sigma_{t}$
and change its color to $J_{t+1}$, and uniformly choose the vertex of
color $\tilde{I}_{t+1}$ in $\tilde{\sigma}_{t}$ and change its color to
$\tilde{J}_{t+1}.$
\end{enumerate}
Finally, the \textit{semi-coordinatewise coupling} of two cyclic dynamics is
defined as follows:
\begin{enumerate}
\setlength\itemsep{3pt}
\item $\min\big(\,\vert S_{t}^{1}-\tilde{S}_{t}^{1}\vert ,\,\vert
S_{t}^{2}-\tilde{S}_{t}^{2}\vert,\,\vert S_{t}^{3}-
\tilde{S}_{t}^{3}\vert\, \big)\geq\frac{2}{n}$: Update the chains
independently.
\item$\min\big(\,\vert S_{t}^{1}-\tilde{S}_{t}^{1}\vert ,\,\vert
S_{t}^{2}-\tilde{S}_{t}^{2}\vert,\,\vert S_{t}^{3}-
\tilde{S}_{t}^{3}\vert\, \big)=\frac{1}{n}$: There exists $i\in\{1,2,3\}$
such that $\vert S_{t}^{i}-\tilde{S}_{t}^{i}\vert =\frac{1}{n}$. Choose a
minimum $i$ value that satisfies the condition, and update the chains
based on $\{i\}$-coordinatewise coupling.
\item$\min\big(\,\vert S_{t}^{1}-\tilde{S}_{t}^{1}\vert ,\,\vert
S_{t}^{2}-\tilde{S}_{t}^{2}\vert,\,\vert S_{t}^{3}-
\tilde{S}_{t}^{3}\vert\,\big)=0$: Find $i\in\{1, 2, 3\}$ such that $\vert S_{t}^{i}-
\tilde{S}_{t}^{i}\vert =0$. We may assume that
$S_{t}^{i+1}\geq\tilde{S}_{t}^{i+1}$.
Then,
\begin{enumerate}
\setlength\itemsep{3pt}
\item Choose the colors $(\,I_{t+1},\, \tilde{I}_{t+1}\,)$ based
on the probability, as stated below:
\begin{equation*}
(\,I_{t+1},\, \tilde{I}_{t+1}\,)\,=
\begin{cases}
(\,i,\,i\,)&\text{ w.p. }
S_{t}^{i}\,=\,\tilde{S}_{t}^{i}\\
(\,i+1,\,i+1\,)&\text{ w.p. } \min\,(\,S_{t}^{i+1},\,\tilde{S}_{t}^{i+1}\,)\\
(\,i+2,\,i+2\,)&\text{ w.p. }
\min\,(\,S_{t}^{i+2},\,\tilde{S}_{t}^{i+2}\,)\\
(\,i+1,\,i+2\,)&\text{ w.p. } S_{t}^{i+1}-\tilde{S}_{t}^{i+1}.
\end{cases}
\end{equation*}
\item Choose the colors $(\,J_{t+1},\, \tilde{J}_{t+1}\,)$ according to the discrete random variable $X$
dependent on $(\,I_{t+1},\, \tilde{I}_{t+1}\,)$, as stated below:
\begin{itemize}[leftmargin=*]
\setlength\itemsep{3pt}
\item$(\,I_{t+1},\, \tilde{I}_{t+1}\,)= (\,i,\,i\,)$: Select $X$ uniformly from  $\{\,(\,i,\,i\,),\,(\,i+1,\,i+1\,)\,\}$.
\item$(\,I_{t+1},\, \tilde{I}_{t+1}\,)=(\,i+1,\,i+1\,)$: Select $X$ uniformly from $\{\,(\,i+1,\,i+1\,),\, (\,i+1,\, i+2\,),\,(\,i+2, \,i+1\,),\,(\,i+2,\,i+2\,)\,\}$.
\item$(\,I_{t+1},\, \tilde{I}_{t+1}\,)=(\,i+2,\,i+2\,)$: Select $X$ uniformly from $\{\,(\,i,\,i\,),\,(\,i+2,\,i+2\,)\,\}$.
\item$(\,I_{t+1},\, \tilde{I}_{t+1}\,)=(\,i+1,\,i+2\,)$: Select $X$ uniformly from $\{\,(\,i+1,\,i\,),\,(\,i+2,\,i+2\,)\,\}$.
\end{itemize}
\item Choose a vertex uniformly that has the color ${I}_{t+1}$ in
${\sigma}_{t}$. Then, change its color to ${J}_{t+1}$ in ${\sigma}_{t+1}$.
\item Choose a vertex uniformly that has the color $\tilde{I}_{t+1}$
in $\tilde{\sigma}_{t}$. Then, change its color to $\tilde{J}_{t+1}$ in
$\tilde{\sigma}_{t+1}$.
\end{enumerate}
The coupling for the case of $S_{t}^{i+1}\leq \tilde{S}_{t}^{i+1}$ can be
similarly defined.
\end{enumerate}
Let $\mathbb{P}_{\sigma_{0}, \tilde{\sigma}_{0}}^{CC}$ be the underlying
measure, and $\mathbb{E}_{\sigma_{0}, \tilde{\sigma}_{0}}^{CC}$ be the
expectation of semi-coordinatewise coupling. This coupling is used to
prove that the $\ell^{1}$-norm between two proportion chains is a supermartingale. It is also used to guarantee the lower bound of its
variance.\par
On Proposition \ref{3.5}, we limit the $\ell^{1}$-norm between $S_{t}$ and
$\tilde{S}_{t}$ with $n^{-1}$ scale. The semi-coordinatewise coupling is used on Proposition \ref{3.14} to make this norm as a supermartingale. Proposition \ref{3.4} provides the bound of the time
spent to limit the norm.
\begin{prop}\label{3.14}
Consider the two cyclic dynamics $(\sigma_{t})$ and
$(\tilde{\sigma}_{t})$, and the corresponding proportion chains $(S_{t})$
and $(\tilde{S}_{t})$. Let $\mathbf{d}_{t} = \Vert S_{t+1}-
\tilde{S}_{t+1}\Vert_{1}-\Vert S_{t}-
\tilde{S}_{t}\Vert_{1}$. Suppose that $\Vert S_{t}-\tilde{S}_{t}\Vert_{1}\geq \frac{10}{n}$ for some $t\geq0$. If semi-coordinatewise coupling is applied in the
following step, then 
\begin{equation*}
\mathbb{E}^{CC}_{\sigma_{0}, \tilde{\sigma}_{0}}\big[\,\mathbf{d}_{t}\,\vert\,\mathcal{F}_{t}\,\big]\leq 0.
\end{equation*}
\end{prop}

Consider the two cyclic dynamics $(\sigma_{t})$ and $(\tilde{\sigma}_{t})$
and its proportion chains. Define the time $T_{1} = \min\,\{\,t:\Vert
S_{t}-\tilde{S}_{t}\Vert_{1}<\frac{10}{n}\,\}$. In the next proposition, we prove that $T_{1}$ is bounded with high probability. Moreover, if the value of the proportion chains $S_{t}$ and $\tilde{S}_{t}$ are bounded with $n^{-\frac{1}{2}}$ scale at time zero,
then the two chains are bounded until time $T_{1}$ with high probability.
\begin{prop}\label{3.5}
Consider the two cyclic dynamics $(\sigma_{t})$ and $(\tilde{\sigma}_{t})$
that satisfy $\sigma_{0}, \tilde{\sigma}_{0}\in
\mathcal{S}^{\frac{r_{0}}{\sqrt{n}}}$ for some $r_{0}>0$. For a fixed
value of $\epsilon>0$, the following statement holds:\par
There exist constant $\gamma, r>0$ such that
\begin{equation*}
\mathbb{P}_{\sigma_{0}, \tilde{\sigma}_{0}}^{CC}\,\Big(\,T_{1}<\gamma
n,\,\max\big(\,\Vert S_{t}-\bar{e}\Vert_{2},\, \Vert\tilde{S}_{t}-
\bar{e}\Vert_{2}\,\big)<\frac{r}{\sqrt{n}}\,\, \forall\, t\leq T_{1}\,
\Big)\,\geq\, 1-\epsilon
\end{equation*}
holds for all sufficiently large $n$ that is bigger than $100r^{2}$.
\end{prop}

Because the $\ell^{1}$-norm between two proportion chains is bounded, the
coalescence of these chains can be proved with semi-synchronized coupling.
\begin{prop}\label{3.6}
Consider the two cyclic dynamics $(\sigma_{t})$ and $(\tilde{\sigma}_{t})$
where $\Vert S_{0}-\tilde{S}_{0}\Vert_{1}<\frac{10}{n}$ holds. For the fixed
constant $\epsilon>0$, there exists a sufficiently large $\gamma > 0$ such
that if $t\geq\gamma n,$
\begin{equation*}
\mathbb{P}^{SC}_{\sigma_{0},
\tilde{\sigma}_{0}}\,\big(\,S_{t}=\tilde{S}_{t}\,\big)\,\geq\,1-\epsilon.
\end{equation*}
\end{prop}

\subsection{Coalescing Basket Chains}\label{3.22}
For the two cyclic dynamics $(\sigma_{t})$ and $(\tilde{\sigma}_{t})$, the
basket chains are made to coalesce with high probability. 
The theorems and proofs
presented throughout this section are similar to \cite{1}; however,
because the conditions are slightly different, detailed analyses are
provided here for completeness.\par
First, define $\mathcal{B}$ as a $3$-partition of
$V_{n}=\{1,\dots,n\}$, and let $\mathcal{B} = (
\mathcal{B}_{m})_{m=1}^{3}$.
We call these $\mathcal{B}_{m}$ partitions as baskets, and denote $\mathcal{B}$ as the $\lambda$-partition if
$\vert\mathcal{B}_{m}\vert>\lambda n$ holds for all $m$. For the
configuration $\sigma\in\Sigma_{n}$, define $3\times3$ matrix
$\mathbf{S}(\sigma)$ representing the proportions of the number of the
vertices in each basket, i.e.
\begin{equation*}
\mathbf{S}^{m, k}\,(\sigma) \,=\,
\frac{1}{\vert\mathcal{B}_{m}\vert}\,\sum_{v\in\mathcal{B}_{m}}\textbf{1}_
{\{\sigma(v) = k\}}\quad m,\, k\,\in\, \{\,1,\, 2, \,3\,\}.
\end{equation*}
The \textit{basket chain} $(\mathbf{S}_{t})$ of the cyclic dynamics $(\sigma_{t})$
is defined as $\mathbf{S}_{t}=\mathbf{S}(\sigma_{t})$. Note that the
basket chain is a Markov chain.\par
Recalling the sets $\mathcal{S}$, $\mathcal{S}^{\rho}$,
$\mathcal{S}^{\rho+}$ in Section \ref{2.21}, $\mathbb{S}$ is defined as
the set of $3\times3$ matrices where each row of each matrix exist in
$\mathcal{S}$. This set is denoted as $\prod_{m=1}^{3}\mathcal{S}$, and
the sets $\mathbb{S}^{\rho}=\prod_{m=1}^{3}\mathcal{S}^{\rho}$ and
$\mathbb{S}^{\rho+}=\prod_{m=1}^{3}\mathcal{S}^{\rho+}$ are similarly
defined. \par
Proposition \ref{3.7} provides the contraction of the basket chains.
Proposition \ref{3.8} limits the distribution of the basket chains.
\begin{prop}\label{3.7}
Suppose that $\mathcal{B}$ is a $\lambda$-partition for some $\lambda>0$.
Consider the basket chain $(\mathbf{S}_{t})$ and the proportion chain
$(S_{t})$ of the cyclic dynamics $(\sigma_{t})$, and define the $3\times3$
matrix $\mathbf{Q}_{t}$ as $\mathbf{Q}^{m,k}_{t}=\mathbf{S}^{m,k}_{t}-
S^{k}_{t}.$ Then, the following holds:
\begin{equation*}
\mathbb{E}_{\sigma_{0}}\big[\,\big(\,\textbf{Q}^{m,k}_{t+1}\,\big)^{2}\,\big]\,=\,\Big(\,1
\,-
\,\frac{1}{n}\,\Big)\,\mathbb{E}_{\sigma_{0}}\big[\,\big(\,\textbf{Q}^{m,k}_{t}\,\big)^{2}\,\big]\,+\,\frac{1}{n}\,\mathbb{E}_{\sigma_{0}}\big[\,\textbf{Q}^{m,k-
1}_{t}\,\textbf{Q}^{m,k}_{t}\,\big]\,+\,O\,\Big(\,\frac{1}{n^{2}}\,\Big
).
\end{equation*}
\end{prop}

\begin{prop}\label{3.8}
For the cyclic dynamics $(\sigma_{t})$, consider the $\lambda$-partition basket $\mathcal{B}$ for some $\lambda > 0$. Assume that either of these conditions are satisfied:
\begin{enumerate}
\setlength\itemsep{3pt}
\item $t\,\geq\, t(n).$
\item $\textbf{S}_{0}\in\mathbb{S}^{\frac{r_{0}}{\sqrt{n}}}$ and
$t\,\leq\,\gamma_{0}n$ for some constant $r_{0}, \gamma_{0} > 0.$
\end{enumerate}
Then, for sufficiently large $r>0$, the following holds for all sufficiently large $n$:
\begin{equation*}
\mathbb{P}_{\sigma_{0}}\,\Big(\,\textbf{S}_{t}\,\notin\,\mathbb{S}^{\frac
{r}{\sqrt{n}}}\,\Big)\,=\,O\,\big(\,r^{-2}\,\big).
\end{equation*}
\end{prop}

\subsubsection{Basketwise Coupling}
The \textit{basketwise coupling} introduced in \cite{1} is utilized herein. The
objective of basketwise coupling is to enable the two basket chains to
coalesce. This coupling is used in Proposition \ref{3.9} to prove that
$\Vert\textbf{S}{}^{m}_{t}-\tilde{\textbf{S}}{}^{m}_{t}\Vert_{1}$ is a
supermartingale. \par
Consider the two cyclic dynamics $(\sigma_{t})$ and
$(\tilde{\sigma_{t}})$, where $S_{0} = \tilde{S}_{0}$. The coupling begins
at $t = 0$, $m = 1$.
While $\textbf{S}{}_{t}^{m}\neq\tilde{\textbf{S}}{}^{m}_{t}$,
\begin{enumerate}
\setlength\itemsep{3pt}
\item Choose the color $I_{t+1} = \tilde{I}_{t+1}$ according to the distribution
$S_{t}= \tilde{S}_{t}$.
\item Choose the color $J_{t+1} = \tilde{J}_{t+1}$ as $I_{t+1} =
\tilde{I}_{t+1}$ with probability $\frac{1}{2}$, and $I_{t+1}+1 =
\tilde{I}_{t+1}+1$ with probability $\frac{1}{2}$.
\item Uniformly choose the vertex $V_{t+1}$ that has the color
$I_{t+1}$ in $\sigma_{t}$.
\item Choose the vertex $\tilde{V}_{t+1}$ based on the following
rules:
\begin{enumerate}
\setlength\itemsep{3pt}
\item If $V_{t+1}\in \mathcal{B}_{m_{0}}$ for some $m_{0}<m$, then
uniformly choose $\tilde{V}_{t+1}$ in $\mathcal{B}_{m_{0}}$ that has the
color $\tilde{I}_{t+1}$ in $\tilde{\sigma}_{t}.$
\item If $V_{t+1}\in \mathcal{B}_{m_{0}}$ for some $m_{0}\geq m$,
$\textbf{S}{}^{m, I_{t+1}}_{t}\neq \tilde{\textbf{S}}{}^{m, \tilde{I}_{t+1}}_{t}$
and $\textbf{S}{}^{m, J_{t+1}}_{t}\neq \tilde{\textbf{S}}{}^{m,
\tilde{J}_{t+1}}_{t}$, then uniformly choose $\tilde{V}_{t+1}$ in
$\mathcal{B}_{[m,3]}$ that has the color $\tilde{I}_{t+1}$ in
$\tilde{\sigma}_{t}$. $\mathcal{B}_{[m,3]}$ is defined as
$\bigcup_{i=m}^{3}\mathcal{B}_{i}$.
\item In other cases, let $\{{v}_{i}\} = v_{1},v_{2}, \dots$ be an
enumeration of the vertices in $\mathcal{B}_{[m, 3]}$ with the color
$I_{t+1}$ in $\sigma_{t}$. It is first ordered based on the index of the
basket it belongs to, and then based on its index in $V$. Let
$\{\tilde{v}_{i}\}=\tilde{v}_{1}, \tilde{v}_{2}, \dots$ be the enumeration
of the vertices in $\mathcal{B}_{[m, 3]}$ with the color $\tilde{I}_{t+1}$
in $\tilde{\sigma}_{t}$ having the same rule. Then, as $V_{t+1} \in
\{v_{i}\}$, there exists $j$ that satisfies $V_{t+1} = v_{j}.$ Let
$\tilde{V}_{t+1}$ as $\tilde{v}_{j}\in\{\tilde{v}_{i}\}$.
\end{enumerate}
\item Change the color of the vertex $V_{t+1}$ to $J_{t+1}$ in
$\sigma_{t}$, and change the color of the vertex $\tilde{V}_{t+1}$ to
$\tilde{J}_{t+1}$ in $\tilde{\sigma}_{t}.$
\end{enumerate}
When $\textbf{S}{}_{t}^{1} = \tilde{\textbf{S}}{}_{t}^{1}$ is reached,
repeat the process with m = 2.
Note that if $\textbf{S}{}_{t}^{1} = \tilde{\textbf{S}}{}_{t}^{1}$ and
$\textbf{S}{}_{t}^{2} = \tilde{\textbf{S}}{}_{t}^{2}$, then
$\textbf{S}{}_{t}^{3} = \tilde{\textbf{S}}{}_{t}^{3}.$
Denote $\mathbb{P}^{BC}_{\sigma_{0},\tilde{\sigma}_{0}}$ as the
underlying probability measure of the coupling, and let
$\mathbb{E}^{BC}_{\sigma_{0},\tilde{\sigma}_{0}}$ be the expectation and 
$\mathbb{V}ar^{BC}_{\sigma_{0},\tilde{\sigma}_{0}}$ be the 
variance with respect to the probability measure. The following proposition proves the coalescence
of the basket chains with high probability.
\begin{prop}\label{3.9}
For the two cyclic dynamics $(\sigma_{t})$ and $(\tilde{\sigma}_{t})$,
suppose that $S_{0}=\tilde{S}_{0}$ and $\textbf{S}_{0},
\tilde{\textbf{S}}_{0}\in \mathbb{S}^{\frac{r}{\sqrt{n}}}$ for some
constant $r > 0$. Let $\mathcal{B}$ be a $\lambda$-partition, where
$\lambda > 0$. Then, for a given $\epsilon > 0$, there exists sufficiently large $\gamma$
that satisfies
\begin{equation*}
\mathbb{P}^{BC}_{\sigma_{0},
\tilde{\sigma}_{0}}\,\big(\,\textbf{S}_{\gamma
n}=\tilde{\textbf{S}}_{\gamma n}\,\big)\,\geq\,1-\epsilon.
\end{equation*}
\end{prop}

Now, the overall coupling, which is a combination of the coupling methods
proposed in the previous sections, is introduced. With this coupling, the
coalescence of the two cyclic dynamics is obtained with high probability,
and the proof of the upper bound is completed.
\subsection{Overall Coupling}\label{3.23}
The \textit{overall coupling} of the two cyclic dynamics $(\sigma_{t})$ and
$(\tilde{\sigma}_{t})$ is denoted with the parameters $\gamma_{1}$,
$\gamma_{2}$, $\gamma_{3}$, $\gamma_{4}\,>\,0$. These parameters are taken
from Proposition \ref{2.2}, Proposition \ref{3.5}, Proposition \ref{3.6}, and
Proposition \ref{3.9}, respectively. The first cyclic dynamics $(\sigma_{t})$ starts at $\sigma_{0}$ and the
second cyclic dynamics $(\tilde{\sigma}_{t})$ starts at
$\tilde{\sigma}_{0}$, where $\tilde{\sigma}_{0}$ is determined according to the distribution
$\mu_{n}$. The coupling is evolved through the following procedure: 
\begin{enumerate}
\setlength\itemsep{3pt}
\item Iterate two chains independently until time
$t_{(1)}(n)=\gamma_{1}n$.
\item Configure the baskets $\mathcal{B} = \bigcup_{k =
1}^{3}\mathcal{B}_{k}$ with the colors in $\sigma_{t_{(1)}(n)}$, i.e.
$\mathcal{B}_{k} =\{\,v:\sigma_{t_{(1)}(n)}(v) = k\,\},\,k = 1,\,2,\,3.$
\item Iterate two chains independently until time
$t_{(2)}(n)=t_{(1)}(n)+t(n).$
\item Iterate two chains with semi-coordinatewise coupling until
time $t_{(3)}(n)=t_{(2)}(n)+\gamma_{2}n.$
\item Iterate two chains with semi-synchronized coupling until time
$t_{(4)}(n)=t_{(3)}(n)+\gamma_{3}n.$
\item Iterate two chains with basketwise coupling until time
$t_{(5)}(n)=t_{(4)}(n)+\gamma_{4}n.$
\end{enumerate}
Denote $\mathbb{P}^{OC}_{\sigma_{0}}$ as the underlying probability
measure.
\subsection{Proof of Upper Bound}\label{3.44}
Here, the proof of the upper bound of the mixing time is demonstrated
using the overall coupling. The proof is similar to that in \cite[Section
4.7]{1}; however, because the conditions are slightly different, the
detailed proof is provided here to ensure completeness.
\begin{prop}\label{3.10}
For the cyclic dynamics $(\sigma_{t})$ and the constant $\epsilon > 0$, and for all sufficiently large $n$,
there exists $\gamma > 0$ such that
\begin{equation*}
\Vert\,\mathbb{P}_{\sigma_{0}}(\sigma_{t_{\gamma}(n)}\in\cdot)-
\mu_{n}\,\Vert_{\textnormal{TV}}\leq\epsilon.
\end{equation*}
\end{prop}
\begin{proof}
First, the overall coupling is applied via seven steps:
\begin{enumerate}
\setlength\itemsep{3pt}
\item Choose $\rho > 0$. By Proposition \ref{2.2}, a sufficiently
large $\gamma_{1}$ satisfying $S_{t_{(1)}(n)}\in \mathcal{S}^{\rho}$ with probability $1-\epsilon/6$ for
all large $n$ can be chosen.
\item Then, $\mathcal{B}$ (defined in step 2 of the overall coupling
procedure) can be considered as a $(\frac{1}{3}-\rho)$ partition.
\item By Proposition \ref{3.1}, for some $r > 0$, $S_{t_{(2)}(n)}\in
\mathcal{S}^{\frac{r}{\sqrt{n}}}$ with a minimum probability of $1-
\epsilon/6$.
\item By the proof of Proposition \ref{2.8}, if $r$ is sufficiently
large, $\tilde{S}_{t_{(2)}(n)}\in \mathcal{S}^{\frac{r}{\sqrt{n}}}$ with a
minimum probability of $1-\epsilon/6$.
\item By Propositions \ref{3.5} and \ref{3.6}, $S_{t_{(4)}(n)}=
\tilde{S}_{t_{(4)}(n)}$ with a minimum probability of $1-3\epsilon/6$.
\item By Proposition \ref{3.8}, for sufficiently large $r_{1}>0$,
$\textbf{S}_{t_{(4)}(n)}, \tilde{\textbf{S}}_{t_{(4)}(n)}\in
\mathbb{S}^{\frac{r_{1}}{\sqrt{n}}}$ with a minimum probability of $1-
\epsilon/6$.
\item By Proposition \ref{3.9}, $\textbf{S}_{t_{(5)}(n)}=
\tilde{\textbf{S}}_{t_{(5)}(n)}$ with a minimum probability of $1-
5\epsilon/6$.
\end{enumerate}
When $t\geq t_{(1)}(n)$ and $\mathcal{F}_{t_{(1)}(n)}$ are given, by the
manner in which the baskets $\mathcal{B}$ were defined, the distribution
of $\sigma_{t}$ is the same under the permutations of the vertices on each
basket $\mathcal{B}_{m}$. As the probability measure $\mu_{n}$ is uniformly distributed in
$\Sigma_{n}$, the same notion can be applied to $\tilde{\sigma}_{t}$.
Thus,
\begin{align*}
&\Vert\,\mathbb{P}^{OC}_{\sigma_{0}}\,\big(\,\sigma_{t_{(5)}(n)}\in
\cdot\,\vert\,\mathcal{F}_{t_{(1)}(n)},\,S_{t_{(1)}(n)}\in \mathcal{S}^{\rho}\,\big)\,-\,
\mu_{n}\,\Vert_{\textnormal{TV}}\\&\quad=\Vert\,\mathbb{P}^{OC}_{\sigma_{0}}\,\big(\,\textbf{S}_{t_{(5)}(n)}\in \cdot\,\vert\,\mathcal{F}_{t_{(1)}(n)},\,S_{t_{(1)}(n)}\in \mathcal{S}^{\rho}\,\big)\,-\,\mu_{n}\circ
\textbf{S}^{-1}\,\Vert_{\textnormal{TV}}\,\\&\quad\leq\,
\mathbb{P}^{OC}_{\sigma_{0}}\big(\,\textbf{S}_{t_{(5)}(n)}\neq\tilde{\textbf{S}}
_{t_{(5)}(n)}\,\vert\,\mathcal{F}_{t_{(1)}(n)},\,S_{t_{(1)}(n)}\in \mathcal{S}^{\rho}\,\big)\,\leq\,5\epsilon/6.
\end{align*}
Therefore, 
\begin{align*}
&\Vert\,\mathbb{P}_{\sigma_{0}}\,(\,\sigma_{t_{(5)}(n)}\in \cdot\,)-
\mu_{n}\,\Vert_{\textnormal{TV}}\\&\quad\leq\,\mathbb{E}^{OC}_{\sigma_{0}}\,\big[\,\Vert\,\mathbb{P}^{OC}_{\sigma_{0}}\,\big(\,\sigma_{t_{(5)}(n)}\in
\cdot\,\vert\,\mathcal{F}_{t_{(1)}(n)}\,\big)\,-\,
\mu_{n}\,\Vert_{\textnormal{TV}}\,\vert\,S_{t_{(1)}(n)}\in \mathcal{S}^{\rho}\,\big]\,+\,\mathbb{P}^{OC}_{\sigma_{0}}\,\big(\,S_{t_{(1)}(n)}\notin \mathcal{S}^{\rho}\,\big)
\,\leq\,\epsilon.
\end{align*}
Finally, let $\gamma$ be $\gamma_{1}+\gamma_{2}+\gamma_{3}+\gamma_{4}$.
This completes the proof.
\end{proof}
\subsection{Proof of Theorem 1.1}
In Theorem \ref{1.1}, the lower bound of the mixing time is guaranteed by
the Proposition \ref{2.8}, and the upper bound of the mixing time is
guaranteed by the Proposition \ref{3.10}. Therefore, the cutoff phenomenon
of cyclic dynamics is proved.\vskip 7pt
\paragraph{\textit{Acknowledgement.}}
The author was supported by the National Research Foundation of Korea (NRF) grant funded by the Korea government (MSIT) (No. 2018R1C1B6006896). This work was supervised by I. Seo.  
\printbibliography
\end{document}